\documentclass[sn-mathphys-num]{sn-jnl}

\usepackage{graphicx}%
\usepackage{multirow}%
\usepackage{amsmath,amssymb,amsfonts}%
\usepackage{amsthm}%
\usepackage{mathrsfs}%
\usepackage[title]{appendix}%
\usepackage{xcolor}%
\usepackage{textcomp}%
\usepackage{manyfoot}%
\usepackage{booktabs}%
\usepackage{algorithm}%
\usepackage{algorithmicx}%
\usepackage{algpseudocode}%
\usepackage{listings}%

\usepackage{enumitem}
\usepackage{cleveref}

\usepackage{mathtools}

\newcommand{\R}{\mathbb{R}}
\newcommand{\N}{\mathbb{N}}

\renewcommand{\d}{\mathsf{d}}

\theoremstyle{thmstyleone}%
\newtheorem{theorem}{Theorem}[section]
\newtheorem{lemma}{Lemma}[section]
\newtheorem{corollary}{Corollary}[section]

\theoremstyle{thmstyletwo}%

\theoremstyle{thmstylethree}%
\newtheorem{definition}{Definition}[section]%

\numberwithin{equation}{section}

\begin{document}

\title[Article Title]{Convergence rates of Newton's method for strongly self-concordant minimization}


\author[1]{\fnm{Nick} \sur{Tsipinakis}}\email{nikolaos.tsipinakis@unidistance.ch}

\author*[2]{\fnm{Panos} \sur{Parpas}}\email{panos.parpas@imperial.ac.uk}

\affil[1]{\orgdiv{Department of Mathematics and Computer Science}, \orgname{UniDistance Suisse}, \orgaddress{ \city{Brig}, \country{Switzerland}}}

\affil*[2]{\orgdiv{Department of Computing}, \orgname{Imperial College}, \orgaddress{\city{London}, \country{UK}}}


\abstract{Newton's method has been thoroughly studied for the class of self-concordant functions. However, a local analysis specific to strongly self-concordant functions---a subclass of the former---is missing from the literature. The local quadratic rate of strongly self-concordant functions follows, of course, from the known results for self-concordant functions. 
However, it is not known whether strongly self-concordant functions enjoy better theoretical properties. In this paper, we study the local convergence of Newton's method for this subclass. We show that its quadratic convergence rate differs from that of general self-concordant functions. In particular, it is provably faster for a wide range of objective functions and benefits from a larger region of local convergence. Thus, the results of this paper close the gap in the theoretical understanding of Newton’s method applied to strongly self-concordant functions.
}




\maketitle

\section{Introduction}
Newton's method is a cornerstone in numerical optimization, renowned for its powerful theory which allows the method to achieve a local quadratic convergence rate \cite{nesterov2018lectures, boyd2004convex}. Its rigorous theoretical basis was particularly pronounced within the framework of self-concordant functions---a class primarily introduced to facilitate the analysis of interior-point methods by Nesterov and Nemirovski in the early 1990s \cite{nesterov1994interior}. A function is called self-concordant if its third derivative is bounded above by the local norm induced by the Hessian matrix. That is,
\begin{equation} \label{def: self concordant func}
|D^3 f(x)[u,u,u]| \leq 2 M_\text{sc} \left[u^T \nabla^2 f(x) u \right]^{3/2},    
\end{equation}
where $M_\text{sc} \geq 0$. This property ensures that the function's curvature does not change too rapidly when $x$ changes, and can replace effectively the most commonly used Lipschitz continuity of the Hessians.

In nonlinear optimization, the benefits of studying Newton's method in this setting had a big impact both in theory and in practice. Under this assumption, it is possible to provide analyses that are invariant to the affine transformation of variables, since assumptions such as the Lipschitz continuity of the Hessians and the strong convexity of $f$, which are usually accompanied by unknown constants, are no longer required. The practical impact of self-concordance is the derivation of meaningful damping parameters when analyzing Newton's global behavior, as they do not depend on such unknowns. 

The benefits of self-concordance are also evident in the analysis of various Newton-type methods. In \cite{tran2015composite}, the proximal Newton method was studied, and both its local quadratic convergence and scale-invariance properties were established, along with global convergence guarantees. The inexact variant of the proximal Newton method was further analyzed in \cite{li2017inexact}, where a scale-invariant linear-to-quadratic convergence rate was derived; under certain assumptions, this rate can become superlinear. Self-concordance has also been incorporated into the analysis of randomized Newton methods. In \cite{pilanci2017newton}, the authors proved linear-quadratic and superlinear convergence rates for the sub-sampled Newton method, with a theoretical framework that does not rely on unknown problem-dependent constants. Furthermore, scale-invariant convergence analyses were extended to multilevel and subspace Newton methods in \cite{tsipinakis2023lowrank, tsipinakis2024multilevel}, where it was shown that superlinear rates can be achieved under specific problem structures. Regarding quasi-Newton methods, self-concordance was incorporated to show global convergence of BFGS without line-search \cite{gao2019quasi}.

The theory of self-concordant functions has been extended to encompass a broader class of objective functions. Quasi self-concordant functions were introduced to analyze the logistic regression loss, which does not satisfy the classical self-concordance property in \eqref{def: self concordant func}. Under this relaxed framework, Newton’s method with gradient regularization was studied and shown to achieve a global linear convergence rate \cite{doikov2023minimizing}. In parallel, the concept of self-concordant-like functions was proposed in \cite{tran2015self-concordant-like} to address composite convex optimization problems. 
It was demonstrated that the proximal Newton method enjoys both global convergence and local quadratic convergence in this setting. A more general framework, termed generalized self-concordant functions, was introduced in \cite{sun2019generalized}. 
Within this framework, Newton's method was proven to exhibit local quadratic convergence near the solution $x^*$, and explicit damping parameters were derived to ensure global convergence from any initial point $x_0 \in \mathbb{R}^n$. Additionally, a class of strongly self-concordant functions was introduced for the analysis of quasi-Newton methods, such as the convex Broyden class. It is a subclass of the classical self-concordance and hence a stronger assumption. Specifically, compared to the general class, strongly self-concordance imposes a stronger assumption on the curvature of the objective function. However, within this framework, the first explicit superlinear convergence rates for the Broyden class were established in \cite{rodomanov2021greedy}, offering deeper insights into the nature of superlinear convergence in quasi-Newton methods.

In light of these developments, the class of self-concordant functions and its extensions, although occasionally strong in their assumptions, appear to be the most appropriate analytical framework for studying the convergence behavior of Newton-type methods. Despite these advances, the convergence behavior of the classical Newton method when applied specifically to strongly self-concordant functions remains underexplored. The local convergence properties of Newton's method follows from existing results. However, whether or not there are any theoretical or practical implications of imposing the extra assumptions has been hitherto unexplored. This paper addresses this gap by conducting a local convergence analysis of Newton's method for the strongly self-concordant setting. We show that the rate of local quadratic convergence in this subclass differs from that in the general case. In particular, we show that Newton’s method converges provably faster for functions that satisfy the strongly self-concordant assumption and that it benefits from a larger region of local quadratic convergence. Thus, our results close a gap in the theoretical understanding of Newton’s method when applied to strongly self-concordant functions and formally confirm the intuition that stronger curvature assumptions lead to improved convergence behavior.

\section{Preliminaries}
For a function $f : \R^n \rightarrow \R$, the gradient and Hessian at $x \in \R^n$ are denoted by $\nabla f(x)$ and $\nabla^2 f(x)$, respectively.
We say that a matrix is $A \in \R^{n \times n}$ is positive semidefinite iff $x^T A x \geq 0 $, for all $x \in \R^n$, and we write $A \succeq 0$. Thus, for $A, B \in \R^{n \times n}$ we write $A \succeq B$ iff $x^T \left(A - B\right) x \geq 0 $, for all $x \in \R^n$. We also define the function $\omega \ : \ \R \to \R$ where
\begin{equation} \label{eq: omega definition}
    \omega(t) \coloneq t - \ln({1+t})
\end{equation}
Let $f$, be a twice differentiable strictly convex functions. We are interested in solving the unconstrained optimization problem
\begin{equation*}
    x^* = \underset{x \in \R^n}{\operatorname{arg \  min}}  \ f(x).
\end{equation*}
Let us fix an $x \in \R^n$. The local norms induced by the Hessian at $x$ are defined as follows
\begin{equation} \label{eq: local norms definition}
    \| y \|_x = \left( y^T \nabla^2f(x) y \right)^{\frac{1}{2}} \quad \text{and}  \quad  \| y \|_x^* = \left( y^T \nabla^2f(x)^{-1} y \right)^{\frac{1}{2}},
\end{equation}
for any $y \in \R^n$. Note that these norms are well-defined due to the strict convexity of $f$.
We also assume that $f$ is bounded from below such that a solution exists, $x^* \in \R^n$. Newton's method convergence for the class of self-concordant functions is measured based on the local norm of the gradient \eqref{eq: local norms definition}, 
\begin{align}
    \lambda(x) & \coloneq \left[\nabla f(x)^T \nabla^2 f(x)^{-1} \nabla f(x)\right]^{1/2} \label{eq: newton decrement}
\end{align}
One can show that if $f$ satisfies \eqref{def: self concordant func}, then the Newton method converges globally, i.e., from any initialization point $x_0 \in \R^n$. 
\begin{theorem} \cite{nesterov2018lectures}[Theorem 5.1.15] \label{thm: newton first phase self concor}
Suppose the sequence $(x_k)_{k \in \N}$ is generated by the damped Newton method $ x_{k+1} = x_k - \frac{1}{1 + M_\text{sc} \lambda(x_k)} \nabla^2 f (x_k)^{-1} \nabla f (x_k)$. Then, for any $k \in \N$, we have
\begin{equation} \label{ineq: newton first phase self concor}
    f(x_{k+1}) \leq f(x_{k}) - \frac{1}{M_\text{sc}^2} \omega(M_\text{sc} \lambda(x_k))
\end{equation}
where $\omega$ is defined in \eqref{eq: omega definition}.
\end{theorem}
Moreover, Newton's method converges quadratically in a neighborhood of $x^*$. The local quadratic rate can be estimated by the local norm of the gradient.
\begin{theorem} \cite{nesterov2018lectures}[Theorem 5.2.2] \label{thm: newton quadratic phase self concor}
Suppose the sequence $(x_k)_{k \in \N}$ is generated by the pure Newton method $ x_{k+1} = x_k - \nabla^2 f (x_k)^{-1} \nabla f (x_k)$. Then, if $\lambda(x_k) < \frac{1}{M_\text{sc}}$, we have
\begin{equation} \label{ineq: newton quadratic phase self concor with Msc}
\lambda(x_{k+1}) \leq \frac{M_\text{sc} \lambda(x_{k})^2}{\left( 1-M_\text{sc}\lambda(x_{k})\right)^2}.
\end{equation}
\end{theorem}
Theorem \ref{thm: newton quadratic phase self concor} shows that the quadratic convergence rate will be activated for some $\lambda(x_{k}) \leq \frac{3 - \sqrt{5}}{2M_\text{sc}}$, otherwise we apply the damped Newton phase which converges as in \eqref{ineq: newton first phase self concor}.

Moreover, we call a function strongly self-concordant according to the following definition.
\begin{definition} \label{def: strongly self concordance}
        A function is strongly self-concordant with constant $M>0$ if for all $x,y,z,w \in \R^n$ it holds
    \begin{equation}
        \nabla^2f(y) - \nabla^2f(x) \preceq M \|x-y\|_z \nabla^2f(w). \label{ineq: strong self conc}
    \end{equation}
\end{definition}

The class of strongly convex self-concordant functions was introduced in \cite{rodomanov2021greedy} to analyze greedy quasi-Newton methods. It is a subclass of the broader class of self-concordant functions. It holds that a strongly self-concordant function is self-concordant with parameter $M_\text{sc} = \frac{M}{2}$ in \eqref{def: self concordant func}.
It has been shown that a function that is strongly convex with constant $\mu$ and has $L$-Lipschitz continuous Hessian is strongly self-concordant with $M = \frac{L}{\mu^{3/2}}$. Nevertheless, unlike strong convexity, strongly self-concordance is an affine invariant property, and thus our analysis is expected to be invariant to affine transformation of variables. Moreover, particular examples of functions satisfying \cref{def: strongly self concordance} include the quadratic function and the log-sum-exp function. The latter has the following form
\begin{equation} \label{eq: example strong self concordance}
    f(x) = \ln \left( \sum_{i=1}^m e^{a_i^T x + b_i} \right) + \frac{1}{2} \sum_{i=1}^m a_i^T x + \frac{\ell_2}{2} \| x\|^2,
\end{equation}
where $a_i \in \R^n, \ell_2>0$ and it has been shown that it is strongly self-concordant with $M=2$ \cite{rodomanov2021greedy}. The class of strongly self-concordant function is of particular interest as it forces the Hessians to be close to each other when $x$ is close to $y$ w.r.t. the local norm, in the sense provided in the following lemma.
\begin{lemma}[\cite{rodomanov2021greedy}, Lemma 4.1] \label{lemma: rodomanov self-concordant}
    Let $f$ be a strongly self-concordant functions with constant $M$, $x,y \in \R^n$, $r \coloneq \| y-x\|_x$ and $G \coloneq \int_0^1 \nabla^2 f (x + t (y - x)) \d t$. Then,
    \begin{align}
        & \frac{1}{1+ Mr} \nabla^2f(x) \preceq \nabla^2f(y) \preceq \left( 1+ Mr \right) \nabla^2f(x) \label{ineq: hessian comparison x_k x_k+1} \\
        & \frac{1}{1+ \frac{Mr}{2}} \nabla^2f(x) \preceq G \preceq \left( 1+ \frac{Mr}{2} \right) \nabla^2f(x) \label{ineq: average hessian for x_k} \\
        & \frac{1}{1+ \frac{Mr}{2}} \nabla^2f(y) \preceq G \preceq \left( 1+ \frac{Mr}{2} \right) \nabla^2f(y) \label{ineq: average hessian for x_k+1}
    \end{align}    
\end{lemma}

\section{Newton for strongly self-concordant functions}
 In this section we study the convergence of the classical Newton method when applied to strongly self-concordant functions. We begin by showing its global convergence. The proof is based on general self-concordant functions in \eqref{def: self concordant func}. Later, we present its local quadratic rates strongly self-concordant functions.

Similar to self-concordant functions, Newton's method convergence will be measured in terms of the Newton decrement. We start by estimating the convergence of the Damped Newton method. Consider the updates
\begin{equation} \label{eq: damped newton}
    x_{k+1} = x_k - \frac{1}{1 + \alpha_k} \nabla^2 f (x_k)^{-1} \nabla f (x_k),
\end{equation}
where $\alpha_k \geq 0$ and $k \in \N$. Since strong self-concordance is a subclass of self-concordance, we can immediately guarantee the descent and global property of the Damped Newton method using the same damping parameter, see \cref{thm: newton first phase self concor}. 
\begin{corollary} \label{cor: damped newton convergence}
    Suppose the sequence $(x_k)_{k \in \N}$ is generated by \eqref{eq: damped newton} with $\alpha_k \coloneq M\lambda(x_k)$. Then,
    \begin{equation*}
         f(x_{k}) - f(x_{k+1}) \geq \gamma \coloneq \frac{1}{4M^2} \omega(2M \lambda(x_k)),
    \end{equation*}
    where $\omega$ is defined in \eqref{eq: omega definition}.
\end{corollary}
 \begin{proof}
     The result follows immediately from inequality \eqref{ineq: newton first phase self concor} and the fact that strong self-concordant functions are self-concordant with parameter $M_\text{sc} = M/2$.
 \end{proof}

The result in \cref{cor: damped newton convergence} holds for any \(x_0 \in \mathbb{R}^n\) and guarantees that the objective function decreases by at least some constant \(\gamma > 0\) at each iteration. The following auxiliary lemma will be useful in our analysis.
\begin{lemma} \label{lemma: quadratic form bound on matrices}
Let $B, G \in \R^{n \times n}$ be symmetric positive definite and $a, b \in \R$ such that
\begin{equation} \label{ineq: ass in lemma appendix}
     a B \preceq G - B \preceq b B.
\end{equation}
Then,
\begin{equation} \label{ineq: quadratic form bound on matrices}
    \left( G - B\right) B^{-1} \left( G - B\right) \preceq c^2 B
\end{equation}
where $c = \max\{|a|, |b| \}$.
\end{lemma}
\begin{proof}
Let $C \coloneq B^{-1/2} \left( G - B\right) B^{-1/2}$. Then by \eqref{ineq: ass in lemma appendix}
    \begin{equation} \label{ineq: C ineq}
        a I \preceq C \preceq b I.
    \end{equation}
    It holds $C^T = C$ and thus $C$ is symmetric positive definite, i.e., $C$ has real eigenvalues. Applying the eigenvalue decomposition on $C$, then by \eqref{ineq: C ineq} we have that $a \leq \lambda_i(C) \leq b$, for all $i = 1,2, \ldots, n$,  where $\lambda_i(C)$ is the $i^\text{th}$ eigenvalue of $C$. Thus, $\lambda_i(C) \leq c$, which implies $\lambda_i(C)^{2} \leq c^2$, for all  $i = 1,2, \ldots, n$. This yields
    \begin{align*}
        C^2 \preceq c^2 I \implies  B^{1/2}C^2 B^{1/2}\preceq c^2 B,  
    \end{align*}
    which is \eqref{ineq: quadratic form bound on matrices}, and thus the proof is finished.
\end{proof}    

We now turn our focus to the local convergence properties of the scheme in \eqref{eq: damped newton}, which is the main contribution of this paper. In what follows, we present our core lemma, which quantifies the progress achieved in a single step of the scheme \eqref{eq: damped newton}. We will need the following definitions.

\begin{equation} \label{eq: def r_k, d_k}
\begin{split}
    r_k \coloneq \|x_{k+1} - x_{k} \|_{x_{k}}, \quad & d_k \coloneq - \frac{1}{1+ \alpha_k} \nabla^2 f(x_{k})^{-1} \nabla f(x_{k}), \\ 
    G_k \coloneq \int_0^1 \nabla^2 & f (x_k  + t (x_{k+1} - x_k)) \d t.
\end{split}
\end{equation}
where the local norm $\|x_{k+1} - x_{k} \|_{x_{k}}$ is defined in \eqref{eq: local norms definition}.

\begin{lemma} \label{lemma: newton's progress lambda}
Suppose the sequence $(x_k)_{k \in \N}$ is generated by the Newton method \eqref{eq: damped newton}. Furthermore, denote 
    \begin{equation} \label{eq: def c_k}
        c_k \coloneq \max \left\lbrace \frac{M r_k }{2 + M r_k} + \alpha_k, \left| \frac{ Mr_k}{2} - \alpha_k \right| \right\rbrace.
    \end{equation}
Then, 
\begin{equation}
    \lambda(x_{k+1}) \leq c_k \sqrt{1 + M r_k} \lambda(x_k) \label{ineq: newton's progress lambda}
\end{equation}
for all $k \in \N$.
\end{lemma}

\begin{proof}
Let us fix $k \in \N$ and denote $ \nabla f_{k} \coloneq\nabla f(x_{k})$, $\nabla^2 f_{k} \coloneq\nabla^2 f(x_{k})$ and $\lambda_{k} \coloneq\lambda(x_{k})$.
First, note from Taylor's identity, $\nabla f_{k+1} = \nabla f_{k} + G_k (x_{k+1} - x_k)$, that
\begin{equation} \label{eq: taylors theorem}
    \nabla f_{k+1} \overset{\eqref{eq: def r_k, d_k}}{=} \nabla f_{k} + G_k d_k \overset{\eqref{eq: def r_k, d_k}}{=} -(1 + \alpha_k) \nabla^2 f_{k} d_k + G_k d_k = \left(G_k - (1 + \alpha_k) \nabla^2 f_{k}\right) d_k.
\end{equation}
Now we have, 
    \begin{equation} \label{ineq: lambda initial}
    \begin{split}
    \lambda_{k+1} & \overset{\eqref{eq: newton decrement}}{=} \left[\nabla f_{k+1}^T \nabla^2 f_{k+1}^{-1} \nabla f_{k+1}\right]^{1/2} \\
    & \overset{\eqref{eq: taylors theorem}}{=} \left[d_k^T \left(G_k - (1 + \alpha_k) \nabla^2 f_{k}\right) \nabla^2 f_{k+1}^{-1} \left(G_k - (1 + \alpha_k) \nabla^2 f_{k}\right) d_k \right]^{1/2} \\
    & \overset{\eqref{ineq: hessian comparison x_k x_k+1}}{\leq} \sqrt{1 + M r_k} 
\left[d_k^T \left(G_k - (1 + \alpha_k) \nabla^2 f_{k}\right) \nabla^2 f_{k}^{-1} \left(G_k - (1 + \alpha_k) \nabla^2 f_{k}\right) d_k \right]^{1/2}
    \end{split}
    \end{equation}
    Further, using \eqref{ineq: average hessian for x_k} and subtracting $(1 + \alpha_k)\nabla^2 f_{k}$ we obtain 
    \begin{align}
       - & \left( \frac{M r_k}{2 + M r_k} + \alpha_k \right) \nabla^2 f_{k} \preceq G_k - (1 + \alpha_k)\nabla^2 f_{k} \preceq \left( \frac{Mr_k}{2} - \alpha_k \right)\nabla^2 f_{k} \notag \\
       & \left( \frac{M r_k}{2 + M r_k} + \alpha_k \right) \frac{1 + a_k}{1 + a_k}\nabla^2 f_{k} \preceq G_k - (1 + \alpha_k)\nabla^2 f_{k} \preceq \left( \frac{Mr_k}{2} - \alpha_k \right) \frac{1 + a_k}{1 + a_k}\nabla^2 f_{k} \nabla^2 f_{k} \label{ineq: bound with c_k}
    \end{align}
    From \eqref{ineq: bound with c_k} and \cref{lemma: quadratic form bound on matrices} for $G\coloneq G_k$, $B \coloneq (1 + \alpha_k)\nabla^2 f_{k}$ and $\tilde{c}_k = c_k / (1+a_k)$ we obtain
    \begin{align}  \label{ineq: quadratic in theorem}
        & (G_k - (1 + \alpha_k)\nabla^2 f_{k})  \frac{\nabla^2 f_k^{-1}}{1+\alpha_k} (G_k - (1 + \alpha_k)\nabla^2 f_{k}) \preceq \tilde{c}_k^2 (1 + \alpha_k)\nabla^2 f_k \notag \\
        \iff & (G_k - (1 + \alpha_k)\nabla^2 f_{k})  \nabla^2 f_k^{-1} (G_k - (1 + \alpha_k)\nabla^2 f_{k}) \preceq c_k^2 \nabla^2 f_k
    \end{align}
    Plugging \eqref{ineq: quadratic in theorem} into \eqref{ineq: lambda initial} we take
    \begin{align*}
        \lambda_{k+1} & \leq c_k \sqrt{1 + M r_k} \left[d_k^T \nabla^2 f_k d_k \right]^{1/2} \\
        & = c_k \sqrt{1 + M r_k} \lambda_k,
    \end{align*}
    as required.
\end{proof}

In what follows we estimate the local convergence of both pure and damped Newton methods using the result of \cref{lemma: newton's progress lambda}.

\begin{theorem} \label{thm: newton's rate}
Suppose the sequence $(x_k)_{k \in \N}$ is generated by the Newton's method \eqref{eq: damped newton}. Then for all $k \in \N$ the following relationships are satisfied:
\begin{enumerate}
    \item If $\alpha_k = 0$, then 
    \begin{equation}
            \lambda(x_{k+1}) \leq  \frac{\sqrt{1 + M \lambda(x_{k})}}{2} M \lambda(x_{k})^2, \label{ineq: pure newtons rate}
    \end{equation}
    \item If $\alpha_k = M \lambda(x_k)$, then 
    \begin{equation} \label{ineq: local rate damped newton}
    \begin{split}
        \lambda(x_{k+1}) & \leq \sqrt{\frac{1 + 2 M \lambda(x_{k})}{1 + M \lambda(x_{k})}} \frac{3 + 3 M \lambda(x_{k})}{2 + 3M \lambda(x_{k})} M \lambda(x_{k})^2  \leq \frac{3 \sqrt{2}}{2} M \lambda(x_{k})^2.
        \end{split}
    \end{equation}
\end{enumerate}
\end{theorem}

\begin{proof}
    Denote $\lambda_{k} \coloneq \lambda(x_{k})$.  By the definition of $r_k$ in \eqref{eq: def r_k, d_k}, note that $r_k = \frac{1}{1+\alpha_k} \lambda_k$. Then we have that:
    \begin{enumerate}
        \item If $\alpha_k = 0$, we get that $r_k = \lambda_k$ and,
        \begin{equation*}
            c_k \overset{\eqref{eq: def c_k}}{=} \max \left\lbrace \frac{M \lambda_k }{2 + M \lambda_k},  \frac{M\lambda_k}{2} \right\rbrace = \frac{M\lambda_k}{2}.
        \end{equation*}
        Putting this all into \eqref{ineq: newton's progress lambda}, inequality \eqref{ineq: pure newtons rate} follows. 
        \item If $\alpha_k = M \lambda_k$, we get that $r_k = \frac{1}{1+M \lambda_k} \lambda_k$ and
        \begin{align*}
            c_k  \overset{\eqref{eq: def c_k}}{=} & \max \left\lbrace \frac{M \lambda_k }{2 + 3M \lambda_k} + M \lambda_k,  \left|\frac{M \lambda_k }{2 + 2M \lambda_k} - M \lambda_k \right| \right\rbrace \\
               = & \max \left\lbrace \frac{3 M^2 \lambda_k^2 +3 M \lambda_k}{2 + 3 M \lambda_k}, \frac{2 M^2 \lambda_k^2 + M \lambda_k}{2 + 2 M \lambda_k} \right\rbrace \\
               = & \frac{3 M^2 \lambda_k^2 +3 M \lambda_k}{2 + 3 M \lambda_k}.
        \end{align*}
            Replacing the above into \eqref{ineq: newton's progress lambda}, and using the fact that $\alpha_k = M \lambda_k$ and $r_k = \frac{1}{1+M \lambda_k} \lambda_k$ we have that 
            \begin{equation} \label{ineq: first local rate damped newton}
                        \lambda(x_{k+1}) \leq \sqrt{\frac{1 + 2 M \lambda_k}{1 + M \lambda_k}} \frac{3+ 3 M \lambda_k }{2 + 3 M \lambda_k} M \lambda_k^2,
            \end{equation}
            and thus the left hand-side in inequality \eqref{ineq: local rate damped newton} is proved. Further, 
            \begin{equation} \label{ineq: first fraction bound}
                \sqrt{\frac{1 + 2 M \lambda_k}{1 + M \lambda_k}} \leq \sqrt{2} \quad \text{and} \quad \frac{3 + 3M \lambda_k}{2 + 3M \lambda_k} \leq \frac{3}{2}.
            \end{equation}
            Combining \eqref{ineq: first local rate damped newton} and \eqref{ineq: first fraction bound}, the right hand-side of inequality \eqref{ineq: local rate damped newton} follows. The proof is completed.
    \end{enumerate}
\end{proof}

Theorem \ref{thm: newton's rate} shows that, between the two local estimates, the one of the pure Newton method \eqref{ineq: pure newtons rate} is more attractive when $ M \lambda(x_{k})$ is small. Theorem \ref{thm: newton's rate} also provides us with the following description of the pure Newton's local convergence. If $\lambda(x_k) \leq \frac{131}{100M}$, then the method will converge quadratically. In particular, 
\begin{equation*}
    \lambda(x_{k+1}) \leq  \frac{\sqrt{1 + M \lambda(x_{k})}}{2} M \lambda(x_{k})^2 \leq 
    \frac{\sqrt{1 + \frac{131}{100}}}{2} \frac{131}{100} \lambda(x_{k}) < \lambda(x_{k}).
\end{equation*}
On the other hand, if $\lambda(x_k) > \frac{131}{100M}$, by \cref{cor: damped newton convergence} and the monotonicity of $\omega(\cdot)$, it holds that
    \begin{equation*}
        f(x_{k+1}) \leq f(x_{k}) - \frac{1}{4M^2} \omega\left(\frac{262}{100}\right).
    \end{equation*}
Hence, the total number of steps before the method enters its quadratic phase are at most
\begin{equation*}
    K = \frac{4M^2}{\omega\left(\frac{262}{100}\right)} \left( f(x_0) - f(x^*)\right)
\end{equation*}
where $x_0 \in \R^n$ is any initialization point in the first phase.

Let us now compare the convergence rates of the Newton method when minimizing strongly self-concordant and self-concordant functions. Recall that a strongly self-concordant function is self-concordant with parameter $M_\text{sc}=M/2$. Then the local quadratic rate of the Newton's method for self-concordant functions with parameter $M_\text{sc}$ is as follows: if $\lambda(x_{k})  \leq \frac{1}{M_\text{sc}} =\frac{2}{M}$, then
\begin{equation} \label{ineq: newton's rate self concordant}
    \lambda(x_{k+1}) \overset{\eqref{ineq: newton quadratic phase self concor with Msc}}{\leq} \frac{\frac{M}{2}\lambda(x_{k})^2}{\left(1-\frac{M}{2}\lambda(x_{k})\right)^2}.
\end{equation}
To say something about the rates \eqref{ineq: newton's rate self concordant} and \eqref{ineq: pure newtons rate}, we will look at a more convenient estimate than \eqref{ineq: pure newtons rate}. In particular, in the proof of \cref{lemma: newton's progress lambda}, one could replace \eqref{ineq: lambda initial} with 
\begin{align*}
    \lambda(x_{k+1}) & \overset{\eqref{eq: newton decrement}}{=} \left[\nabla f_{k+1}^T \nabla^2 f_{k+1}^{-1} \nabla f_{k+1}\right]^{1/2} \\
    & \overset{\eqref{eq: taylors theorem}}{=} \left[d_k^T \left(G_k - (1 + \alpha_k) \nabla^2 f_{k}\right) \nabla^2 f_{k+1}^{-1} \left(G_k - (1 + \alpha_k) \nabla^2 f_{k}\right) d_k \right]^{1/2} \\
    & \overset{\eqref{ineq: average hessian for x_k+1}}{\leq} \sqrt{1 + \frac{M r_k}{2}}
\left[d_k^T \left(G_k - (1 + \alpha_k) \nabla^2 f_{k} \right) G_k^{-1} \left(G_k - (1 + \alpha_k) \nabla^2 f_{k}\right) d_k \right]^{1/2} \\
& \overset{\eqref{ineq: average hessian for x_k}}{\leq} \left(1 + \frac{M r_k}{2}\right)
\left[d_k^T \left(G_k - (1 + \alpha_k) \nabla^2 f_{k}\right) \nabla^2 f_{k}^{-1} \left(G_k - (1 + \alpha_k) \nabla^2 f_{k}\right) d_k \right]^{1/2}\\
& \leq \left(1 + \frac{M r_k}{2}\right) c_k \lambda(x_{k}),
\end{align*}
where the last inequality follows immediately from  \cref{lemma: newton's progress lambda}. Then from \cref{thm: newton's rate} with $\alpha_k=0$, we have 
\begin{equation}
        \lambda(x_{k+1}) \leq \left(1 + \frac{M\lambda(x_k)}{2}\right)\frac{M \lambda(x_k)^2}{2}. \label{ineq: newton rate no 2}
\end{equation}
Inequality \eqref{ineq: newton rate no 2} also shows a local quadratic convergence rate, however it is more conservative than \eqref{ineq: pure newtons rate}. Nevertheless, we can use it to tell something about the behavior the pure Newton method when applied to strong self-concordant functions. In view of the estimates \eqref{ineq: newton rate no 2} and \eqref{ineq: newton's rate self concordant}, and since 
\begin{equation*}
    \left(1 + \frac{M\lambda(x_k)}{2} \right)    \left(1 - \frac{M\lambda(x_k)}{2} \right)^2 = \left(1 - \left(\frac{M\lambda(x_k)}{2}\right)^2 \right) \left(1 - \frac{M\lambda(x_k)}{2} \right) \leq 1,
\end{equation*}
we conclude that the local quadratic rate in \eqref{ineq: newton rate no 2} is faster than the one in \eqref{ineq: newton's rate self concordant}, and thus the pure Newton method is expected to converge faster for strongly self-concordant functions. In summary we obtain 
\begin{equation*}
\begin{split}
    \lambda(x_{k+1}) & \leq  \frac{M\sqrt{1 + M \lambda(x_{k})}}{2} \lambda(x_{k})^2 \leq \left(1 + \frac{M\lambda(x_k)}{2}\right) \frac{M \lambda(x_k)}{2}^2 \\ 
    & \leq \frac{1}{\left(1-\frac{M}{2}\lambda(x_{k})\right)^2} \frac{M\lambda(x_k)^2}{2}.
    \end{split}
\end{equation*}

Regarding the region of the local rate, for self-concordant functions, the Newton method enters its quadratic rate if $\lambda \leq \frac{3 - \sqrt{5}}{M}$, or approximately $\lambda \leq \frac{7}{10M}$, whereas, for strongly self-concordant function the local quadratic rate is activated if $\lambda(x_k) \leq \frac{6}{5M}$. Therefore, minimizing strongly self-concordant functions we obtain an extended region of the local quadratic convergence rate compared to self-concordant functions. 

\section{Conclusions}
We analyze the behavior of Newton’s method when applied to the subclass of strongly self-concordant functions and establish that it exhibits improved convergence properties compared to the general self-concordant setting. In particular, we show that Newton’s method converges faster under this regime and that it also enjoys an expanded region of local quadratic convergence.  These results close an existing gap in the theoretical understanding of the convergence of Newton’s method under the stronger curvature assumptions imposed by the strongly self-concordant setting.


\section*{Acknowledgment}The authors would like to thank Dr. Michael Tsingelis for reading a preliminary version of this paper and providing valuable comments that helped improve its clarity and presentation.

\bibliography{sn-bibliography}

\end{document}